\documentclass[12pt, twoside, leqno]{article}

\usepackage{amsmath,amsthm}
\usepackage{amssymb}

\usepackage{enumitem}

\usepackage{graphicx}

\usepackage[T1]{fontenc}
\usepackage[utf8]{inputenc}
\usepackage{amsmath,amssymb,amsthm,latexsym,graphicx}
\usepackage[normalem]{ulem}
\usepackage{setspace} 
\usepackage{hyperref}
\usepackage{cancel}
\RequirePackage{doi}
\usepackage{hyperref}
\usepackage{doi}

\usepackage{titlesec}

\usepackage{url}

\newcommand{\RR}{{\mathbb R}}

\newcommand{\ZZ}{{\mathbb Z}}

\newcommand{\bc}{\begin{center}}
\newcommand{\ec}{\end{center}}
\newcommand{\be}{\begin{enumerate}}
\newcommand{\ee}{\end{enumerate}}
\newcommand{\bi}{\begin{itemize}}
\newcommand{\ei}{\end{itemize}}

\newcommand{\beq}{\begin{equation}}
\newcommand{\eeq}{\end{equation}}
\newtheorem{exmp}{Example}
\newcommand{\bex}{\begin{exmp}}
	\newcommand{\eex}{\end{exmp}}

\newcommand{\pplus}{{p_{+}}}
\newcommand{\pminus}{{p_{-}}}

\newcommand{\qplus}{{q_{+}}}
\newcommand{\qminus}{{q_{-}}}

\newcommand{\pdot}{{p(\cdot)}}
\newcommand{\qdot}{{q(\cdot)}}

\newcommand{\lpdot}{\ell^{\pdot}}
\newcommand{\lqdot}{\ell^{\qdot}}

\newcommand{\Sol}[1] {\textbf{Solution:}}

\newcommand {\abs}[1] {\lvert#1\rvert}

\newcommand{\normb}[1]{\ensuremath{\left\lVert#1\right\rVert }}

\newcommand\frightarrow{\scalebox{1}[.3]{$\rule[.45ex]{2ex}{1.5pt}%
		\kern-.2ex{\blacktriangleright}$}}

\newcommand{\vertiii}[1]{{\left\vert\kern-0.25ex\left\vert\kern-0.25ex\left\vert #1 
		\right\vert\kern-0.25ex\right\vert\kern-0.25ex\right\vert}}
		
\let\RR\Reals

\let\ZZ\Ints

\newtheorem{theorem}{Theorem}[section]

\newtheorem{lemma}[theorem]{Lemma}

\theoremstyle{definition}
\newtheorem{definition}[theorem]{Definition}
\newtheorem{remark}[theorem]{Remark}

\numberwithin{equation}{section}

\begin{document}


\baselineskip=17pt


\title{The boundedness of Fractional Hardy-Littlewood maximal operator on variable $\lpdot(\ZZ)$ spaces using Calderon-Zygmund decomposition}

\author{A.Sri Sakti Swarup\\
Department of Mathematics\\ 
 Birla Institute of Technology and Science, Jawahar Nagar\\
Hyderabad-500 078, Telangana, India\\
E-mail: p20180442@hyderabad.bits-pilani.ac.in
\and 
A. Michael Alphonse\\
Department of Mathematics\\ 
 Birla Institute of Technology and Science, Jawahar Nagar\\
Hyderabad-500 078, Telangana, India \\
E-mail: alphonse@hyderabad.bits-pilani.ac.in}

\date{May 19, 2022}

\maketitle



\let\thefootnote\relax\footnotetext{2020 \emph{Mathematics Subject Classification}: Primary 42B35; Secondary 42B20.}

\let\thefootnote\relax\footnotetext{Calderon-Zygmund decomposition, Log Holder continuity, fractional Hardy-Littlewood maximal operator, variable sequence spaces. }





\begin{abstract}
In this paper, we prove strong type, weak type inequalities of Hardy-Littlewood maximal operator and fractional Hardy-Littlewood maximal operator on variable sequence spaces $\lpdot (\ZZ)$. This is achieved using Calderon-Zygmund decomposition for sequences, properties of modular functional and Log Holder continuity.
\end{abstract}

\section{Introduction}\label{sec1}
In this paper, we prove strong type, weak type inequalities of Hardy-Littlewood maximal operator and fractional Hardy Littlewood maximal operator on variable $\lpdot (\ZZ)$ sequence spaces. Inequalities for Hardy Littlewood maximal operator for variable sequence spaces $\lpdot (\ZZ)$, follows from corresponding inequalities for $\ell^p(\ZZ)$ spaces and Log Holder continuity. To prove, the inequalities for fractional Hardy Littlewood maximal operator, we use Calderon Zygmund decomposition for sequences and Log Holder continuity. \\
Several authors have proved strong type and weak type inequalities for Hardy-Littlewood maximal operator and fractional Hardy-Littlewood maximal operator using various methods. In this paper, we provide a method that uses Calderon-Zygmund decomposition for sequences
 to prove strong and weak type inequalities for these maximal operators. Similar proof for the real line can be found in $\text{\cite{fior_new_proof_maximal}}$ where boundedness of Hardy-Littlewood maximal operator is proved using Calderon-Zygmund decomposition. 
Even though, we follow same method of proof given in $\text{\cite{fior_new_proof_maximal}}$, there are some remarkable differences 
between the proofs on integers and proofs on real line. \\
For detailed history of variable Lebesgue spaces, please refer $\text{\cite{fior_var_book}},  \text{\cite{lars1}}$. Similar conclusions in broader scope utilizing framework of homogeneous spaces may be found in $\text{\cite{shukla}}$. For detailed references on work in spaces of homogeneous type under variable exponent setting, please refer  $\text{\cite{alme1}}, \text{\cite{mizuta1}}, \text{\cite{goro1}}, \text{\cite{hast2}},\text{\cite{kok1}}, \text{\cite{kok2}}$.

In this paper, we use the method of Calderon Zygmund decomposition for sequences and Log Holder continuity. We assume only Log Holder continuity at infinity. The local Log Holder continuity follows automatically in case of integers.

\section{Definitions and Notation}

In this paper, the following notation are used. Given a bounded sequence $ \left \{ p(n): n \in \ZZ \right \}$ which takes values in $[1, \infty ) $, define 
$\lpdot (\ZZ)$ to be set of all sequences $\left \{ a(k) \right \}$ on $\ZZ$ such that for some $\lambda>0$, $ \sum_{ k \in \ZZ} ( \frac{\abs{a(k)}}{\lambda}   )^{p(k)}< \infty$. \\
Throughout this paper, $ \left \{ p(n): n \in \ZZ \right \}$ denotes a bounded sequence which takes values in $[1,\infty)$.  Define $ \pminus = \inf \left \{ p(n): n \in \ZZ \right \}, \pplus = \sup \left \{ p(n): n \in \ZZ \right \} $. \\
Let $\mathcal{S}$ denote the set of all bounded sequences which take values in $[1, \infty)$ such that $\pplus < \infty $. \\
For any set $ A \subset \ZZ$ , $\abs{A}$ denotes cardinality of A. 
\\
Let 
\[ \Omega = \left \{ n \in \ZZ: 1 \leq p(n) < \infty \right \} \] 
\[ \Omega_\infty = \left \{ n \in \ZZ: p(n) = \infty \right \} \] Then define modular functional associated with $\pdot$ as
\[ \rho_{\pdot}(a) = \sum_{k \in \ZZ \setminus \Omega_\infty} \abs{ a(k)}^{p(k)} + \normb{a}_{\ell^\infty(\Omega_\infty)} \]

Note, when $\pplus < \infty$, modular functional becomes
\[ \rho_{\pdot}(a) = \sum_{k \in \ZZ } \abs{ a(k)}^{p(k)}  \]
Throughout this paper, we assume $\pplus < \infty$. Further for a given sequence $\left \{ a(k): k \in \ZZ \right \}$, define norm in $\ell^{\pdot}(\ZZ)$ as
 \[ \normb{a}_{\ell^{\pdot}(\ZZ)}  = \inf \left \{ \lambda >0: \rho_{\pdot}(\frac{a}{\lambda}) \leq 1 \right \} \]
$\normb{}_{\lpdot(\ZZ)}$ is a norm $\text{\cite{han}}$. A similar norm on $\ell^{\pdot}(\RR^n)$ is defined in $\text{\cite{fior_var_book}}$ and there it is proved it is a norm. 
 
\begin{definition}
Given a non-negative sequence of real numbers, $ \left \{ a(n): n \in \ZZ  \right \}, 0 \leq \alpha<1$, define fractional Hardy-Littlewood Maximal operator as follows:
\[ M_{\alpha} a(n) = \sup_{n \in I} \abs{I}^{\alpha -1} \sum_{k \in I} \abs{a(k)} \] where supremum is taken over all intervals of integers which contain $n$. When $\alpha=0$, fractional Hardy-Littlewood maximal operator becomes Hardy-Littlewood maximal operator.
\end{definition}
 
\begin{definition}
Given a  sequence $\pdot \in \mathcal{S}$,  we say that $\pdot$ is log-Holder continuous at infinity, if there exists positive real constants $C_\infty, p_\infty$ such that for all $n \in \ZZ$, 
\[ \abs{ p(n) - p_\infty } \leq \frac{ C_\infty} { \log(e + \abs{n})} , n \in \ZZ \] and $e$ is exponential number. In this case, we write $\pdot \in LH_\infty(\ZZ)$
\end{definition}

We will also require Young's inequality in this paper. For any real $ a, b \in \RR$, and any two real numbers $p, q$ such that $\frac{1}{p} + \frac{1}{q}=  1$, Young's inequality $\text{\cite{Krey1}}$ is as follows: 
\[ ab \leq  \frac{a^{p}}{p} + \frac{b^{q}}{q } \]

\begin{lemma}{\label{sak_lem_LH_infty}}
Let $ \left \{ q(n) \right \}$ be the sequence which satisfies $\frac{1}{p(n)}  + \frac{1}{q(n)} =1$  $\forall n \in \ZZ$.  Let $1 \leq \pminus \leq p(n) \leq \pplus < \infty$,  $\forall n \in \ZZ$. The following are equivalent:
\begin{enumerate}
\item $\pdot \in LH_\infty(\ZZ)$
\item $\frac{1}{\pdot} \in LH_\infty(\ZZ)$ 
\item $\frac{1}{\qdot} \in LH_\infty(\ZZ)$ 
\item $\qdot \in LH_\infty(\ZZ)$ 
\end{enumerate}
\end{lemma}
\begin{proof}
(a) We shall prove (1) $\implies (2).\\
 $Let $\pdot \in LH_\infty(\ZZ)$. Note, that when $\pplus < \infty$,  $\forall n \in \ZZ$,
\begin{align*}
\abs{ \frac{1}{p(n)} - \frac{1}{p_\infty}} = \abs{ \frac{p(n) - p_\infty}{p(n)p_\infty}} \leq  \frac{\abs{p(n) - p_\infty}}{\pminus p_\infty} \leq \frac{\frac{C_\infty}{\pminus p_\infty}}{\log(e+\abs{n})}  
\end{align*} 
for some  $LH_\infty$ constant, which is $k_\infty = \frac{C_\infty}{\pminus p_\infty}$. \\
(b) We shall prove (2) $\implies$ (1).\\
Let $\frac{1}{\pdot} \in LH_\infty(\ZZ)$. Then,  $\forall n \in \ZZ$
\begin{align*}
\abs{p(n)-p_\infty} &=\abs{p(n)p_\infty \biggl( \frac{1}{p(n)} - \frac{1}{p_\infty} \biggr) } \leq \abs{p(n)p_\infty}  \frac{C_\infty}{\log(e+\abs{n})}  \\
& \leq \pplus p_\infty   \frac{C_\infty  }{\log(e+\abs{n})} \\
& \leq  \frac{k_\infty}{\log(e+\abs{n})}
\end{align*} for some $LH_\infty$ constant which is  $k_\infty = \pplus p_\infty C_\infty $ .\\
(c) We shall prove that (1) $\implies$ (3). \\
Let $\pdot \in LH_\infty(\ZZ)$. Note, that when $\pplus < \infty$, for $n \in \ZZ$,
\begin{align*}
\abs{ \frac{1}{q(n)} - \frac{1}{q_\infty}} &= \abs{ \frac{1}{p(n)} - \frac{1}{p_\infty}} = \abs{ \frac{p(n) - p_\infty}{p(n)p_\infty}} \leq  \frac{\abs{p(n) - p(\infty)}}{\pminus p_\infty} \\
& \leq \frac{\frac{C_\infty}{\pminus p_\infty}}{\log(e+\abs{n})}  
\end{align*} 
for some  $LH_\infty$ constant, which is $k_\infty = \frac{C_\infty}{\pminus p_\infty}$. \\
(d) We shall prove (3) $\implies$ (4)  and (4) $\implies$ (1)\\
 Let $\frac{1}{\qdot} \in LH_\infty(\ZZ)$. Then $\frac{1}{\frac{1}{\qdot}} \in LH_\infty(\ZZ)$, which implies that $\qdot \in LH_\infty(\ZZ)$. This shows (3) $\implies$ (4). \\
(4) $\implies$ (1) follows same argument as (c) . \\
%
\end{proof}

\begin{lemma}
Let $ \left \{ q(n) \right \}$ be the sequence which satisfies $\frac{1}{p(n)}  + \frac{1}{q(n)} =1$  $\forall n \in \ZZ$.  Let $1 \leq \pminus \leq p(n) \leq \pplus < \infty$,  $\forall n \in \ZZ$. Then, $\pdot \in LH_\infty(\ZZ)$ implies $p_\infty \geq \pminus$.
\end{lemma}
\begin{proof}
 Given that $\pdot \in LH_\infty(\ZZ)$ implies $\frac{1}{\pdot} \in LH_\infty(\ZZ)$ using Lemma[$\ref{sak_lem_LH_infty}$], it follows that
\begin{align*}
\abs{ \frac{1}{p_\infty}} & = \abs{ \frac{1}{p(n)} + \frac{1}{p_\infty} - \frac{1}{p(n)} } \\
& \leq \abs{ \frac{1}{p(n)}} + \abs{ \frac{1}{p_\infty} - \frac{1}{p(n) } }  \\
& \leq \abs{\frac{1}{\pminus}} + \frac{C_\infty}{\log(e+\abs{n}) } 
\end{align*}
Since this is true for every $n$, $p_\infty \geq \pminus$.
\end{proof}

The following lemma will be used as a variation of Holder's inequality. Proof of this lemma in continuous version can be found in $\text{\cite{fior_var_book}}$. Same line of proof works here.
\begin{lemma}{\label{sak_holder_lem1}}
For $0 \leq \alpha < 1$, and $ p,q $ such that $ 1< p< \frac{1}{\alpha}, \frac{1}{p} - \frac{1}{q} = {\alpha}$.
For every interval $I$ in $\ZZ$ and non-negative sequence $\left \{ a(k) \right \}$ \\
\[ \abs{I}^{{\alpha} -1 } \sum_{k \in I} a(k) \leq  \biggl( \sum_{k \in I} a(k)^p  \biggr)^{\frac{1}{p} - \frac{1}{q}}  \biggl( \frac{1}{\abs{I}} \sum_{k \in I} a(k)  \biggr)^{\frac{p}{q}} \]
\end{lemma}

\section{Some Results on Variable Lebesgue spaces}
In this section, we prove certain results that are used in proving boundedness of maximal operators.The proofs of these results in discrete case are same as corresponding results on real line $\text{\cite{fior_var_book}}$. For completeness sake, some of the proofs are included. \\

The modular functional defined in Section 2 has the following properties defined below $\text{\cite{fior_var_book}}$.
\begin{lemma}{\label{sak_modular}}
Let $ \left \{ u(n) \right \} $ be a non-negative sequence of real numbers. Let $\pdot \in \mathcal{S}$, then: 
\begin{enumerate}
\item For all u, $\rho_{\pdot}(u) \geq 0$ and $\rho_{\pdot}(\abs{u})=\rho_{\pdot}(u) $
\item $\rho_{\pdot}(u) =0$ if and only if $u(k)=0$ for all k $\in \ZZ$ 
\item If $\rho_{\pdot}(u) < \infty$, then $u(k) < \infty $ for all k $\in \ZZ$ 
\item $\rho_{\pdot}$ is convex: Given $\alpha, \beta \geq 0, \alpha + \beta =1,  \rho_{\pdot}(\alpha u + \beta v) \leq \alpha \rho_{\pdot}(u) + \beta \rho_{\pdot}(v) $
\item If $\abs{u(k)} \geq \abs{v(k)}$ , then $\rho_{\pdot}(u) \geq \rho_{\pdot}(v) $. 
\item If for some $\delta > 0 , \rho_{\pdot}(\frac{u}{\delta})< \infty $, then the function $\lambda \to \rho_{\pdot}(\frac{u}{\lambda})$ is continuous 
and decreasing on $[\delta, \infty)$. Further $\rho_{\pdot}(\frac{u}{\lambda}) \to 0$ as $\lambda \to \infty $.
\end{enumerate}
\end{lemma}


The following lemma gives the connection between modular functional and $\lpdot(\ZZ)$ norm.
\begin{lemma} 
Let $ \left \{ u(n) , n \in \ZZ \right \} $ be a non-negative sequence of real numbers. Let $\pdot \in \mathcal{S}$. 
Then $u \in \lpdot(\ZZ)$ if and only if
\[ \rho_{\pdot}(u) = \sum_{k \in \ZZ} {\abs{u(k)}}^{p(k)}  < \infty \]
\begin{proof}
If $\rho_{\pdot}(u) < \infty$, then by definition of norm $u  \in \ell^{\pdot}(\ZZ)$. Conversely since $u \in \lpdot(\ZZ)$ by definiton of norm, $\rho_{\pdot}(\frac{u}{\lambda}) < \infty$ for some $\lambda >0$. Further by (6) of modular functional properties, we have $\rho_{\pdot}(\frac{u}{\lambda}) < \infty$ for some $\lambda >1$. So, it follows that
\[ \rho_{\pdot}(u)= \sum_{k \in \ZZ} \left( \frac{\abs{u(k)}\lambda}{\lambda} \right)^{p(k)}   \leq \lambda^{\pplus } \rho_{\pdot}(\frac{u}{\lambda}) < \infty  \]
\end{proof}
\end{lemma}

\begin{lemma}{\label{sak_Prop_2.10}}
Let $ \left \{ a(n) , n \in \ZZ \right \} $ be a non-negative sequence of real numbers  such that $a \in \lpdot(\ZZ)$ and let $p(\cdot) \in \mathcal{S}$.
\begin{enumerate}
\item For all $\lambda \geq 1$, 
\[ \lambda^{\pminus} \rho_{\pdot}(a) \leq \rho_{\pdot}(\lambda a) \leq \lambda^{\pplus} \rho_{\pdot}(a) \] 
\item When $0 < \lambda <1$ the reverse inequalities are true, i.e 
\[ \lambda^{\pplus} \rho_{\pdot}(a) \leq \rho_{\pdot}(\lambda a) \leq \lambda^{\pminus} \rho_{\pdot}(a) \]
\end{enumerate}
\end{lemma}
\begin{proof}
For $\lambda \geq 1$ , 
\begin{align*}
\rho_{\pdot}(\lambda a) = \sum_{n=1}^\infty \biggl( \lambda a(n) \biggr)^{p(n)} = \sum_{n=1}^\infty {\lambda}^{p(n)} {a(n)}^{p(n)} \leq \lambda^{\pplus} \rho_{\pdot}(a)
\end{align*}
Further,
\begin{align*}
{\lambda}^{\pminus} \rho_{\pdot}(a) = \sum_{n=1}^{\infty} {\lambda}^{\pminus} {a(n)}^{p(n)} \leq \sum_{n=1}^\infty  \biggl( \lambda a(n) \biggr)^{p(n)} = \rho_{\pdot}(\lambda a)
\end{align*}
Similar proof can be used for the case $0< \lambda < 1$.
\end{proof}

\begin{lemma}{\label{sak_lem7}}[Fatou Property of the Norm].
Let $ u \in \lpdot(\ZZ)$ be sequence of non-negative real numbers and $\pdot \in \mathcal{S}$.
 Further, let $ \left \{ u_k \right \} \subset \ell^{\pdot}(\ZZ)$ be non-negative sequences of real numbers such that $u_k$ increases to the sequence $u$ pointwise. 
Then $ \normb{u_k}_{\lpdot(\ZZ)} \to \normb{u}_{\lpdot(\ZZ)}$.
\end{lemma}
\begin{proof}
Since for every $n$, $\frac{u_k(n)}{\lambda} \leq \frac{u_{k+1}(n)}{\lambda}$, by property (5) of modular functional,
$\rho_{\pdot}(\frac{u_k}{\lambda}) \leq \rho_{\pdot}(\frac{u_{k+1}}{\lambda}) $ and hence $\normb{u_k}_{\lpdot(\ZZ)} \leq \normb{u_{k+1}}_{\lpdot(\ZZ)}$.
Therefore, $\left \{ u_k \right \}$ is an increasing sequence, so is $ \left \{ \normb{u_k}_{\lpdot(\ZZ)} \right \}$ and so this increasing sequence either converges to a finite limit or diverges to $\infty$.\\
It is required to prove $ \lim_{ k \to \infty} \normb{u_k}_{\lpdot(\ZZ)} = \normb{u}_{\lpdot(\ZZ)}$. Note that
$\normb{u_k}_{\lpdot(\ZZ)}$ is increasing and $\normb{u_k}_{\lpdot(\ZZ)} \leq \normb{u}_{\lpdot(\ZZ)}$. Hence
$\lim_{k \to \infty} \normb{u_k}_{\lpdot(\ZZ)} \leq \normb{u}_{\lpdot(\ZZ)}$.

%

Take $\lambda >0, \normb{u}_{\lpdot(\ZZ)} > \lambda$. We shall prove that if $\normb{u}_{\lpdot(\ZZ)} > \lambda$ then for sufficiently large values of $k$, $\normb{u_k}_{\lpdot(\ZZ)} > \lambda$.
Since $\rho_{\pdot}(\frac{u}{\lambda}) >1$ and using Monotone convergence theorem, 
\begin{align*} \rho_{\pdot}(\frac{u}{\lambda}) & =  \sum_{m \in \ZZ}  \biggl( \frac{ {u(m) }} {\lambda } \biggr)^{p(m)}  = 
\sum_{ m \in \ZZ} \biggl ( \frac{\lim_{k \to \infty} u_k(m) }{\lambda} \biggr)^{p(m)} \\
&= \lim_{k \to \infty} \biggl(  \sum_{m \in \ZZ}  \biggl(  \frac{ { u_k(m) } } { \lambda } \biggr)^{p(m)} \biggr) = \lim_{k \to \infty} \rho_{\pdot}(\frac{u_k} {\lambda}) 
\end{align*}
So,  $\rho_{\pdot}(\frac{u_k}{\lambda}) >1$ for sufficiently large values of $k$.
Let $ A_k = \left \{  \lambda >0 :  \rho_{\pdot}(\frac{u_k}{\lambda}) \leq 1 \right \} $ and $ B = \left \{  \lambda >0 :  \rho_{\pdot}(\frac{u}{\lambda}) \leq 1 \right \} $.
From above discussion, $ B^\complement \subseteq {A^{\complement}_k} $ for sufficiently large values of $k$.
Therefore $A_k \subseteq B$ for sufficiently large values of $k$. Hence $ \inf A_k \geq \inf B$ for sufficiently large values of $k$.
Therefore $\normb{u_k}_{\lpdot(\ZZ)}  \geq \normb{u}_{\lpdot(\ZZ)} $ for sufficiently large values of $k$.
and hence $\lim_{k \to \infty} \normb{u_k}_{\lpdot(\ZZ)}  \geq \normb{u_k}_{\lpdot(\ZZ)} $. 
\end{proof}

\begin{lemma}{\label{sak_lem8}}[Fatou's lemma for sequences].
Let $ \left \{ u_k \right \}$ be a non-negative sequence of real numbers such that $ \left \{ u_k \right \} \in \lpdot(\ZZ)$. 
Let $\pdot \in \mathcal{S}$, suppose the sequence $ \left \{ u_k \right \}  \in \ell^{\pdot}(\ZZ) $ such that $u_k(n) \to u(n) $ for every $n$. If
\[ \liminf_{k  \to \infty } \normb{u_k}_{\lpdot(\ZZ)}  < \infty \]
then $ u \in \ell^{\pdot}(\ZZ) $ and 
\[ \normb{u}_{\lpdot(\ZZ)} \leq \liminf_{k \to \infty } \normb{u_k}_{\lpdot(\ZZ)} \]
\end{lemma}

\begin{proof}
Define,
\[ v_k(i) = \inf_{ m \geq k} {u_m(i)} \]
Then for all $ m \geq k, v_k(i) \leq {u_m(i) }$ and this shows that $ v_k \in \ell^{\pdot}(\ZZ)$. Since $ \left \{ v_k \right \}$ is an increasing sequence and
\[ \lim_{k \to \infty } v_k(i) = \liminf_{ m \to \infty}  {u_m(i)} =  {u(i)}, \quad i \in \ZZ \]
Also,
\[ \normb{u}_{\lpdot(\ZZ)} = \normb{ \lim_{k \to \infty} v_k }_{\lpdot(\ZZ)} = \lim_{k \to \infty} \normb{v_k}_{\lpdot(\ZZ)}  =  \liminf_{k \to \infty} \normb{v_k}_{\lpdot(\ZZ)}  \]
Therefore, by Fatou's norm property[$\ref{sak_lem7}$] for  sequences
\[ \normb{u}_{\lpdot(\ZZ)} = \lim_{k \to \infty} \normb{v_k}_{\lpdot(\ZZ)} \leq \lim_{k \to \infty} ( \inf_{ m \geq k} \normb{u_m}_{\lpdot(\ZZ)} ) = \liminf_{k \to \infty} \normb{u_k}_{\lpdot(\ZZ)} \]
So, if $ \liminf_{k \to \infty} \normb{u_k}_{\lpdot(\ZZ)}  < \infty$, then $\normb{u}_{\lpdot(\ZZ)} < \infty$, which implies $ u \in \lpdot(\ZZ)$.
\end{proof}

The following lemma shows normalizing a sequence gives $\rho_{\pdot}(\frac{a}{\normb{a}_{\lpdot(\ZZ)  }}) \leq 1$. Additionally if $\pplus < \infty$, $\rho_{\pdot}(\frac{a}{\normb{a}_{\lpdot(\ZZ)  }}) = 1$.
\begin{lemma}
Let $ \left \{ a(k) \right \} $ be a non-negative sequence of real numbers and $\pdot \in \mathcal{S}$
\begin{enumerate}
\item If $ a \in \ell^{p(\cdot)}({\ZZ})$ and $ \normb{a}_{\lpdot(\ZZ)} > 0$, then $\rho_{\pdot}(\frac{a}{\normb{a}_{\lpdot(\ZZ)  }}) \leq 1$ \\
\item If $\pplus < \infty$, then $ \rho_{\pdot}(\frac{a}{\normb{a}_{\lpdot(\ZZ)  }}) =1$ for all nontrivial $a \in \ell^{p(\cdot)}({\ZZ})$
\end{enumerate}
\end{lemma}

\begin{proof}
(1) By definition$\normb{a}_{\lpdot(\ZZ)} = \inf \left \{ \lambda > 0: \rho_{\pdot}(a/\lambda) \leq 1 \right \}$, $\normb{a}_{\lpdot(\ZZ)}+ \frac{1}{n}$ is not an infimum. Therefore, there exists a $\lambda_n$ such that
$\lambda_n \leq \normb{a}_{\lpdot(\ZZ)} + \frac{1}{n}$ and $ \rho_{\pdot}(\frac{a}{\lambda_n}) \leq 1$.
Fix such a decreasing sequence $\left \{ \lambda_n \right \}$ such that $ \left \{ \lambda_n \right \} \to \normb{a}_{\lpdot(\ZZ)}$. Then by Fatou's lemma and the definition of modular functional,
\begin{align*}
\rho_{\pdot}(\frac{a}{\normb{a}_{\lpdot (\ZZ)}}) & = \sum_{k \in \ZZ} \biggl( \frac{\abs{a(k)}}{\normb{a}_{\lpdot(\ZZ) }} \biggr)^{p(k)}\\
& = \sum_{k \in \ZZ} \lim_{n \to \infty} \biggl( \frac{\abs{a(k)}}{\lambda_n} \biggr)^{p(k)} \\
& \leq \liminf_{n \to \infty}   \sum_{k \in \ZZ}  \biggl( \frac{\abs{a(k)}}{\lambda_n} \biggr)^{p(k)} \\
& = \liminf_{n \to \infty} \rho_{\pdot}(\frac{a}{\lambda_n}) \leq 1
\end{align*}
(2). Assume $\pplus < \infty$ but $\rho(\frac{a}{\normb{a}_{\lpdot(\ZZ)}}) < 1$. Then $\forall \lambda$ such that $0< \lambda < \normb{a}_{\lpdot(\ZZ)}$, by Lemma[\ref{sak_Prop_2.10}]
\[ \rho_{\lpdot(\ZZ)}(a/\lambda) = \rho_{\pdot}( \frac{ \normb{a}_{\lpdot(\ZZ)}}{\lambda} \frac{a}{\normb{a}_{\lpdot(\ZZ)}}) \leq \biggl( \frac{\normb{a}_{\lpdot(\ZZ)}}{\lambda} \biggr)^{\pplus} \rho_{\pdot}(\frac{a}{\normb{a}_{\lpdot(\ZZ)}}) \]
So,a $\lambda$ can be found, sufficiently close to $\normb{a}_{\lpdot(\ZZ)}$ such that $\rho_{\pdot}(a/\lambda)<1$. 
However norm definition requires that $\rho_{\pdot}(a/\lambda) \geq 1$. 

\end{proof}

Using Lemma[$\ref{sak_lem8}$], properties of modular functional and homogenity of the norm, following lemma can be proved. Continuous version of Lemma[$\ref{sak_corollary_2}$] can be found in $\text{\cite{fior_var_book}}$.
\begin{lemma}{\label{sak_corollary_2}}
Let  $\pdot \in \mathcal{S}$ and $ \left \{ a(k) \right \} $ be a non-negative sequence of real numbers such that $ \left \{ a(k) \right \} \in \lpdot(\ZZ)$
\begin{align*}
\text{If} \quad \normb{a}_{\lpdot(\ZZ)} \leq 1, \text{then} \qquad \rho_{\pdot}(a) \leq \normb{a}_{\lpdot(\ZZ)} \\
\text{If} \quad \normb{a}_{\lpdot(\ZZ)} \geq 1, \text{then} \qquad \rho_{\pdot}(a) \geq \normb{a}_{\lpdot(\ZZ)}
\end{align*}
\end{lemma}


Using Lemma[$\ref{sak_Prop_2.10}$], we can prove the following lemma below. Continuous version of Lemma[$\ref{sak_corollary_3}$] can be found in $\text{\cite{fior_var_book}}$. Same line of proof works here.
\begin{lemma}{\label{sak_corollary_3}}
Let $ \left \{ a(k) \right \} $ be a non-negative sequence of real numbers such that $ \left \{ a(k) \right \} \in \lpdot(\ZZ)$  and $\pdot \in \mathcal{S}$. Then
\begin{enumerate}
\item If $ \normb{a}_{\lpdot(\ZZ)} >1$, then $\rho_{\pdot}(a)^{1/\pplus} \leq \normb{a}_{\lpdot(\ZZ)} \leq \rho_{\pdot}(a)^{1/\pminus}$ \\
\item If $ 0 < \normb{a}_{\lpdot(\ZZ)} \leq 1 $, then $\rho_{\pdot}(a)^{1/\pminus} \leq \normb{a}_{\lpdot(\ZZ)} \leq \rho_{\pdot}(a)^{1/\pplus}$ \\
\end{enumerate}
\end{lemma}

\section{Calderon-Zygmund decompostion for Sequences}
\begin{theorem}{\label{CD_General}}
Let $1\leq p < \infty$ and a $\in l^p(\ZZ)$. For every $t>0$, and $0\leq \alpha<1$, there exists a sequence of disjoint intervals $\left \{ I_j^t \right \}$ such that
\begin{align*}
(i)& \quad t < \frac{1}{ \abs{I_j^t}^{1-\alpha}   } \sum_{k \in I_j^t} \abs{a(k)} \leq 2t,\forall j \in \ZZ \\
(ii)& \quad \forall n  \not\in \cup_j I_j^t, \quad \abs{a(n} \leq t \\
(iii)& \quad \text{If} \quad t_1 > t_2 , \quad \textbf{then each} \quad I_j^{t_1} \quad \textbf{is subinterval of some} \quad I_m^{t_2}, \quad \forall j,m \in \ZZ
\end{align*}
\end{theorem}
\begin{proof}
For each positive integer N, consider the collection of disjoint intervals of cardinality $2^N$, \\
\begin{eqnarray*}
\left \{ I_{N,j} \right \} = \left \{ [ (j-1)2^N +1, \dots, j2^N] \right \} , j \in \ZZ.
\end{eqnarray*} \\
For each $t>0$, let $N=N_t$ be the smallest positive integer such that 
\[ \frac{1}{\abs{I_{N_t,j}     }^{1-\alpha}   } \sum_{k \in I_{N_t,j}} \abs{a(k)} \leq t \] 	\\
Such $N_t$ is possible as $ a\in \ell^p (\ZZ)$. Now consider collection $ \left \{ I_{N_t,j} \right \}$ and subdivide each of these intervals into two intervals of equal cardinality. If $I$ is one of these intervals either 
\begin{align*}
(A)& \quad \frac{1}{\abs{I}^{1-\alpha} } \sum_{k \in I} \abs{a(k)} >t \\  \quad \text{or}  \quad
(B)& \quad \frac{1}{\abs{I}^{1-\alpha} } \sum_{k \in I} \abs{a(k)} \leq t 
\end{align*}
In case (A) we select this interval and include it in a collection $\left \{ I_{r,j} \right \}$. \\
In case (B) we subdivide I once again unless I is a singleton and select intervals as above. 
Now the elements which are not included in $\left \{ I_{r,j} \right \}$ form a set S such that for every $n \in S$, $\abs{a(n)} \leq t$. This proves (i).  \\
Also from the choice of $\left \{ I_{r,j} \right \}$, note that $\left \{ I_{r,j} \right \}$ are disjoint and satisfy 
\begin{eqnarray*}
\frac{1}{\abs{I_{r,j}   }^{1-\alpha} } \sum_{k \in I_{r,j}} \abs{a(k)} >t
\end{eqnarray*}
Since each $I_{r,j}$ is contained in an interval $J_0$ with card $J_0 = \abs{2I_{r,j} }^{1-\alpha}$, which is not selected in the previous step, we have
\[
\frac{1}{\abs{I_{r,j}}^{1-\alpha} } \sum_{K \in I_{r,j}} \abs{a(k)} \leq \frac{2}{\abs{J_0 }} \sum_{k \in J_0} \abs{a(k)} \leq 2t.
\]
This proves (ii). It remains to prove (iii). \\
If $t_1 > t_2$ then $N_{t_1} \leq N_{t_2}$. So each $I_{N_{t_1}, j}$ is contained in some 
$I_{N_{t_2}, j}$. In the subdivision and the selecting process for $t_1$ we have 
\[
\frac{1}{\abs{I_{N_{t_1,j}  } }^{1-\alpha} } \sum_{k \in I_{N_{t_1,j} }} \abs{a(k)} > t_1 > t_2.
\]
So, if $I_j^{t_1}$ is not one of the intervals $I_m^{t_2}$, then it must be subinterval of some $I_m^{t_2}$ selected in an earlier step. This completes proof.
\end{proof}

\begin{theorem}{\label{CD_Weak}}
Let $1\leq p < \infty$ and a $\in l^p(\ZZ)$. Let $\left \{ I_j^t \right \}$ be intervals obtained from Calderon Zygmund decomposition at height $t$ and $ 0\leq \alpha <1 $. Then
\[ \left \{ n: M_{\alpha}a(n) > 9t \right \} \subseteq \cup_j 2I_j^t \]
\end{theorem}

\begin{proof}
We apply Calderon-Zygmund decomposition to the sequence $\left \{ a(n), n \in \ZZ \right \} $, where $a \in \ell^p(\ZZ)$. For every $t > 0$, we get a sequence of disjoint intervals $ \left \{ I^t_j \right \}$ satisfying the criteria of Calderon - Zygmund decompostion. 
For each $j$, we have
\[ \frac{1}{{\abs{I^t_j}}^{1-\alpha}} \sum_{k \in I^t_j} \abs{a(k)} > t \]
Therefore, $\cup_j I^t_j \subseteq \left \{n : M_{\alpha} a(n) >t \right \}$.
 Let $n \notin \cup_j 2 I^t_j$ and $I$ be any interval which contains $n$. Then 
\[ \sum_{k \in I} \abs{a(k)} = \sum_{ k \in I \cap (\cup_j I^t_j) } \abs{a(k)} +  \sum_{ I \cap (\cup_j I^t_j)^{\complement} } \abs{a(k)}  = S_1 + S_2 \]
To estimate $S_1$, we observe a simple geometric fact. If $ I \cap I^t_j$ is non-empty  and I is not contained in $ 2 I^t_j$, then $I^t_j \subset 4I$. Since $n \in I$ and $n \notin 2 I^t_j$, for each $j$, I is not contained in $ 2I^t_j$ for each $j$. Also, note that $S_2 \leq \abs{I}$.
Therefore, 
\begin{align*}
S_1 &\leq \sum_{ \left \{ j: I^t_j \subseteq 4I \right \} }  \sum_{ k \in I^t_j} \abs{a(k)} \\
& \leq \sum_{ \left \{ j: I^t_j \subseteq 4I \right \} } 2t  \abs{I^t_j }\\
& \leq 2t \abs{4I} \\
& \leq 8t \abs{I}.  
\end{align*} 
Hence, $ \sum_{ k \in I} \abs{a(k)} \leq S_1 + S_2 \leq 9t \abs{I} $. 

Since $I$ was an arbritrary interval containing $n$, we have $ M_{\alpha}a(n) \leq 9t $. Therefore
\[ ( \cup_j 2 I^t_j )^{\complement} \subseteq \left \{ n: M_{\alpha}a(n) \leq 9t \right \} \]
\end{proof}

\begin{lemma}{\label{CD_Template}}
Let $1\leq p < \infty, a \in l^p(\ZZ)$ and $\qdot \in \mathcal{S}$. Let $\left \{ I^k_j \right \}$ be intervals obtained from Calderon-Zygmumd decomposition at height $(9t)^{k-1}$ and $t >0$. For any maximal operator $M_{\alpha}, 0 \leq \alpha < 1$, and for every $0<t< 1/9$,
\[\sum_{i \in \ZZ} {M_{\alpha} a(i)}^{q(i)}  \leq A^{\qminus}2^{\qplus (\alpha-1)} \sum_{k,j} \sum_{E^k_j} \biggl( \frac{1}{\abs{2I^k_j}^{1-\alpha}} \sum_{i \in {2I^k_j}} a(i) \biggr)^{q(i)} \]
where $A$ is chosen based on $t$.
\end{lemma}
\begin{proof}
Let $A \in \RR$ be such that $A =  9t < 1 $ for every $t > 0$. Define $\Omega_{k} = \left \{ i \in \ZZ: M_{\alpha}a(i) > A^{k} =(9t)^{k} \right \}$. 
For each integer $k$, apply Theorem[$\ref{CD_General}$]  Calderon-Zygmund decomposition for sequence $\left \{ a(i) \right \}$, at height $t =A^{k-1}$ to get
pairwise disjoint cubes $ \left \{ I_j^k \right \} $ such that 
\begin{align}
\Omega_k \subset  \cup_j 2 I_j^k \label{eq:eq1}\tag{CD1}  \\
\frac{1}{ {\abs{I^k_j} }^{1-\alpha} } \sum_{{i \in I^k_j}} a(i) > A^{k-1}   \label{eq:eq2}\tag{CD2}
\end{align}

From \eqref{eq:eq2} we get \\
\begin{align*}
\frac{1}{\abs{2I_j^k}^{1-\alpha}} \sum_{i \in 2I_j^k} a(i)  > \frac{1}{2^{1-\alpha}} A^{k-1} 
\end{align*}

Define sets inductively as follows: 
\begin{align*}
E_1^k = \biggl( \Omega_{k+1} \setminus \Omega_{k} \biggr) \cap 2I_1^k \nonumber \\
E_2^k = \biggl( \biggl( \Omega_{k+1} \setminus \Omega_{k} \biggr) \cap 2I_2^k \biggr) \setminus E_1^k  \nonumber\\
E_3^k = \biggl( \biggl( \Omega_{k+1} \setminus \Omega_{k}\biggr) \cap 2I_3^k  \setminus (E_1^k \cup E_2^k ) \biggr)  \nonumber\\
\ldots
E_m^k = \biggl( \biggl( \Omega_{k+1} \setminus \Omega_{k}\biggr) \cap 2I_{m}^k  \setminus (E_1^k \cup E_2^k  \ldots E_{m-1}^k) \biggr)  \nonumber\\
\end{align*}
Then, sets $E_j^k$ are pairwise disjoint for all $j$ and $k$ and satisfy for every $k$
\[\Omega_{k+1} \setminus \Omega_{k} = \bigcup_j E_j^k \label{eq:eq4} \tag{CD3} \]
So, $\ZZ = \bigcup_k \biggl( \Omega_{k+1}\setminus \Omega_{k} \biggr) =  \bigcup_k \bigcup_j E_j^k $.
Further, note $\Omega_{k+1} \setminus \Omega_{k} = \left \{ A^{k+1} < M_{\alpha} a(i) < A^k \right \}$. \\
We now estimate $ M_{\alpha}a$ as follows, noting $A < 1$,
\begin{align*}
    \sum_{i \in \ZZ} {M_{\alpha} a(i)}^{q(i)}
    & = \sum_k \sum_{\Omega_{k+1} \setminus \Omega_{k}} {M a(i)}^{q(i)} \\
    & \leq \sum_k \sum_{\Omega_{k+1} \setminus \Omega_{k}} [A^{k}]^{q(i)} \\
    & \leq A^{\qminus}2^{\qplus(\alpha-1)} \sum_{k,j} \sum_{E^k_j} \biggl( \frac{1}{\abs{2I^k_j}^{1-\alpha}} \sum_{i \in 2I^k_j} a(i) \biggr)^{q(i)}
\end{align*}
\end{proof}

\section{Fractional Hardy-Littlewood Maximal operator}

\subsection{Strong $(\pdot, \pdot)$ Inequality for Fractional Hardy-Littlewood Maximal Operator}
In this section, we derive strong $(\pdot, \pdot)$ inequality for fractional Hardy-Littlewood maximal Operator. In order to prove this theorem, we use Lemma[$\ref{CD_Template}$].
 
 \begin{theorem}[Strong $(\pdot, \pdot)$ Inequality for Fractional Hardy-Littlewood Maximal Operator]
Given a non-negative sequence $ \left \{ a(i) \right \} \in \lpdot(\ZZ), 0\leq \alpha < 1$, let $ \pdot \in \mathcal{S},\pplus< \infty,  \pminus> 1 , \pdot \in LH_\infty(\ZZ), \qdot \in LH_\infty(\ZZ), \text{and} \quad 1< \pminus \leq \pplus < \frac{1}{\alpha} $. Define the exponent function $\qdot$ by 
\[ \frac{1}{p(k)} - \frac{1}{q(k)} = \alpha ,  k \in \ZZ\] Then
\[ \normb{M_{\alpha} a}_{\ell^{q(.)}(\ZZ)} \leq C \normb{a}_{\ell^{p(.)}(\ZZ)} \]
\end{theorem}

\begin{proof}

Now, we are going to prove $ \normb{M_{\alpha} a}_{\lqdot(\ZZ)} \leq c \normb{a}_{\lpdot(\ZZ)}$. We may assume without loss of generality that $\normb{a}_{\lpdot(\ZZ)} =1$. We will show that there exist a constant $\lambda_2=\lambda_2(\pdot) > 0$ such that $\rho_{\qdot}(M_{\alpha}\frac{a}{\lambda_2}) \leq 1$. 
For this it suffices to prove $ \rho_{\qdot}( \alpha_2 \beta_2 \gamma_2 \delta_2 M_{\alpha} a) \leq \frac{1}{2}$ form some non-negative real numbers $\alpha_2,\beta_2, \gamma_2,\delta_2 $.
Let $ {\lambda_2}^{-1} = \alpha_2 \beta_2 \gamma_2 \delta_2 $.  Then 
\[ \rho_{\qdot}( \alpha_2 \beta_2 \gamma_2 \delta_2 M_{\alpha} a) = \sum_{m \in \ZZ} [\alpha_2 \beta_2 \gamma_2 \delta_2 M_{\alpha} a(m)]^{q(m)} \]
To estimate this term we perform 
Calderon Zygmund decomposition for sequences $\left \{ a(k) \right \}$ and use Lemma[$\ref{CD_Template}$] for the fractional Hardy-Littlewood maximal operator $M_{\alpha}$. We will show  that  $\rho_{\qdot}( \alpha_2 \beta_2 \gamma_2 \delta_2 M_{\alpha} a)  \leq \frac{1}{2} $ for suitable choices of $ \alpha_2, \beta_2, \gamma_2, \delta_2$.
If we set $\alpha_2 = A^\qminus 2^{\qplus{(\alpha-1)}}$ and using Lemma[$\ref{CD_Template}$] for $M_{\alpha}$, we get
\[\sum_{m \in \ZZ} [ \alpha_2 \beta_2 \gamma_2 \delta_2 M_{\alpha} a(m) ]^{q(m)} \leq \sum_{k,j} \sum_{E^k_j} \biggl( \beta_2 \gamma_2 \delta_2 \abs{2I^k_j}^{\alpha-1} \sum_{r \in 2I^k_j} a(r) \biggr)^{q(m)}  \label{eq:eq3}\tag{CD3}  \]
 We have to estimate right hand side of $\eqref{eq:eq3}$. At this point, we note that $q_\infty < \infty$. Let $ g_2(r) = {a(r)}^{p(r) } $, then $\eqref{eq:eq3}$ becomes,
\[ \sum_{k,j} \sum_{E^k_j} \biggl( \beta_2 \gamma_2 \delta_2 \abs{ 2I^k_j}^{\alpha -1} \sum_{r \in 2I^k_j} {g_2(r)}^{\frac{1}{p(r)}}  \biggr)^{q(m)} \]

Note, since $\frac{1}{\pdot} \in LH_\infty(\ZZ)$, it follows that the exponents $p_\infty, q_\infty$ satisfy $ \frac{1}{p_\infty} - \frac{1}{q_\infty} = \alpha $. Hence, using Lemma $\ref{sak_lem1}$  with exponents $p_\infty, q_\infty$.  
\begin{align*}
& \abs{2I^k_j}^{\alpha -1} \sum_{ r \in 2I^k_j} {g_2(r)}^{\frac{1}{p(r)}}  
& \leq \biggl( \sum_{r \in 2I^k_j} {g_2(r)}^{\frac{ p_\infty } { p(r)}} \biggr)^{ \frac{1}{p_\infty} - \frac{1}{q_\infty} } 
\biggl( \frac{1}{\abs{2I^k_j} }\sum_{r \in 2I^k_j} {g_2(r)}^{\frac{1}{p(r)}} \biggr)^{\frac{p_\infty}{q_\infty} } \label{eq:eq4}\tag{CD4}  \\
\end{align*}

Further, note the following estimates.
\begin{enumerate}
\item 
${g_2(r)}^{p_\infty} \leq 1$. since $g_2(r)^{p_\infty} =  {a(r)}^{p(r) p_\infty} \leq 1$ since $ a(r) \leq 1$ \\

\item 
Using $R(k) = (e+ \abs{k})^{-N}$, with $N >1$. By taking $N$ large enough, it follows that
\[ \sum_{m \in \ZZ} {R(m)}^{\frac{1}{p_\infty}} \leq \sum_{m \in \ZZ} {R(m)}^{\frac{1}{q_\infty}} \leq 1 \]
By taking $N$ large, the sum i.e $ \sum_{m \in \ZZ}  R(m) \leq 1 $. \\

\item Also, note that 
\[ \sum_{r \in 2I^k_j} g_2(r) \leq \sum_{r \in 2I^k_j} a_2(r)^{p(r)} \leq \sum_{r \in \ZZ} a(r)^{p(r)} \leq \normb{a}_{\lpdot(\ZZ)} =1 \] \\

\item  Since ${g_2(k)}^{p_\infty} \leq 1$. Put $ F = {g_2(k)}^{p_\infty} $, so that $F \leq 1$. Hence, we use following form of Lemma[$\ref{sak_lem1}$] , where $\frac{1}{r(\cdot)} $ is taken as $LH_\infty$ constant.
\[ \sum_{m \in \ZZ} F(m)^{\frac{1}{p(m)}} \leq C \sum_{m \in \ZZ} F(m)^{\frac{1}{p_\infty}}  + C \sum_{m \in \ZZ} {R(m)}^{\frac{1}{p_\infty}} \]
Since $ \frac{1}{\pdot} \in LH_\infty(\ZZ) $ and using Lemma[$\ref{sak_lem1}$] with exponents $p_\infty, q_\infty$ with $ F= {g_2(r)}^{p_\infty} \leq 1$, based on estimates [1-4], it follows that
\begin{align*}
\sum_{\ZZ} \biggl( g_2(k)^{p_\infty} \biggr)^{\frac{1}{p(m)} } & \leq C \sum_{m \in \ZZ} \biggl( g_2(k)^{p_\infty} \biggr)^{\frac{1}{p_\infty}} + C \sum_{m \in \ZZ} R(m)^{\frac{1}{p_\infty}} \\
&= C \sum_{m \in \ZZ} g_2(k) + C \sum_{m \in \ZZ} R(m)^{\frac{1}{p_\infty}} \leq C
\end{align*}
Therefore $\eqref{eq:eq4}$ is
\begin{align*}
 \abs{2I^k_j}^{\alpha -1} \sum_{ r \in 2I^k_j} {g_2(r)}^{\frac{1}{p(r)}}  
& \leq \biggl( \sum_{r \in 2I^k_j} {g_2(r)}^{\frac{ p_\infty } { p(r)}} \biggr)^{ \frac{1}{p_\infty} - \frac{1}{q_\infty} } \biggl( \frac{1}{\
\abs{2I^k_j} }\sum_{r \in 2I^k_j} {g_2(r)}^{\frac{1}{p(r)}} \biggr)^{\frac{p_\infty}{q_\infty} } \\
& \leq C^{\alpha} \biggl( \frac{1}{\abs{2I^k_j} }\sum_{r \in 2I^k_j} {g_2(r)}^{\frac{1}{p(r)}} \biggr)^{\frac{p_\infty}{q_\infty} }
\end{align*}

\end{enumerate}

Therefore, using estimates [1-4], we can choose constant $\beta_2 >0 $  such that
\begin{align*}
\sum_{k,j} \sum_{E^k_j} & \biggl( \beta_2 \gamma_2 \delta_2 \abs{ 2I_j}^{\alpha-1} {g_2(r)}^{\frac{1}{p(r)}} \biggr)^{q(m)}  \leq \sum_{k,j} \sum_{E^k_j} \gamma_2 \delta_2 \biggl(  \biggl( \frac{1}{\abs{2I^k_j}} \sum_{r \in 2I^k_j} {g_2(r)}^{\frac{1}{p(r)}} \biggr)^{p_\infty} \biggr)^{\frac{q(m)}{q_\infty} }
\end{align*}

Note, $ \frac{1}{\abs{2I^k_j}} \sum_{r \in 2I^k_j} {g_2(r)}^{\frac{p_\infty} { p(r)} } \leq \frac{1}{\abs{2I^k_j}} \sum_{r \in \ZZ} {g_2(r)}^{\frac{p_\infty} { p(r)} } \leq 1 $. \\
So, let $ F = \biggl(  g_2(k)^{\frac{1}{p(k)}} \biggr)^{q(k) p_\infty} \leq \biggl( \sum_{r \in \ZZ} g_2(k)^{\frac{1}{p(k)}} \biggr)^{q(k) p_\infty} \leq 1$. \\
 Using Lemma[$\ref{sak_lem1}$] and with $\frac{1}{\qdot} \in LH_\infty(\ZZ)$, we get
\[ \sum_{k \in \ZZ}  F^{\frac{1}{q_\infty}} \leq  C_1 \sum_{k \in \ZZ}  F^{\frac{1}{q(k)}}  + C_2 \sum_{k \in \ZZ}  R^{\frac{1}{q_\infty}} \]
Hence,
\begin{align*}
F^{\frac{1}{q_\infty}} &\leq \sum_{k \in \ZZ}  F^{\frac{1}{q_\infty}} \leq C_1 \sum_{k \in \ZZ}  F^{\frac{1}{q(k)}} + C_2 \sum_{k \in \ZZ}  R^{\frac{1}{q_\infty}} \\
& = C_1 \biggl( \sum_{k \in \ZZ}  g_2(k)^{\frac{1}{p(k)}} \biggr)^{p_\infty}  + C_2 \sum_{k \in \ZZ}  R^{\frac{1}{q_\infty}} 
\end{align*}

Therefore,
\begin{align*}
\biggl( \sum_{k \in \ZZ}  {g_2(k)}^{\frac{1}{p(k)}} \biggr)^{q(k) \frac{p_\infty}{q_\infty} } & \leq C_1 \biggl( \sum_{k \in \ZZ}  {g_2(k)}^{\frac{1}{p(k)}} \biggr)^{p_\infty} + C_2 \sum_{k \in \ZZ} {R(k)}^{\frac{1}{q_\infty}} \\
&\leq C_1 \biggl( \sum_{k \in \ZZ}  {g_2(k)}^{\frac{1}{p(k)}} \biggr)^{p_\infty} + C_2 \sum_{k \in \ZZ} {R(k)}^{\frac{1}{q_\infty}} 
\end{align*}
and so,
\begin{align*}
\sum_{k,j} \sum_{E^k_j}  \gamma_2 \delta_2 & \biggl( \biggl( \frac{1}{\abs{2I^k_j}} \sum_{k \in 2I^k_j} g_2(k)^{\frac{1}{p(k)}} \biggr)^{p_\infty} \biggr)^{\frac{q(m)}{q_\infty}} \\
&\leq \delta_2  \biggl( C_1 \biggl( \frac{1}{\abs{2I^k_j}} \sum_{k \in 2I^k_j} {g_2(k)}^{\frac{1}{p(k)}} \biggr)^{p_\infty}  + C_2 \sum_{k \in \ZZ} {R(k)}^{\frac{1}{q_\infty}} 
\end{align*}

Take $\gamma_2>0$ such that $C_1 =1, C_2 = \frac{1}{6}$. Then
\begin{align*}
\sum_{k,j} \sum_{E^k_j} & \gamma_2 \delta_2 \biggl( \biggl( \frac{1}{\abs{2I^k_j}} \sum_{k \in 2I^k_j} g_2(k)^{\frac{1}{p(k)}} \biggr)^{p_\infty} \biggr)^{\frac{q(m)}{q_\infty}}  \\
&\leq \sum_{k,j} \sum_{E^k_j} \delta_2 \biggl( \biggl (\frac{1}{\abs{2I^k_j}} \sum_{k \in 2I^k_j} g_2(k)^{\frac{1}{p(k)}} \biggl)^{p_\infty} \biggr) + \frac{1}{6} \sum_{k \in \ZZ} {R(k)}^{\frac{1}{q_\infty}}  \\
& \leq \sum_{\ZZ} \delta_2 {M g_2(\cdot)}^{\frac{1}{\pdot}} (k)^{p_\infty} + \frac{1}{6}  
\end{align*}

Note that the maximal operator is bounded on $\ell^{p_\infty} (\ZZ)$, since $p_\infty \geq \pminus > 1$.
Again apply Lemma[$\ref{sak_lem1}$] to get \\
\[ \sum_{k \in \ZZ} M( g_2(\cdot)^{\frac{1}{\pdot}} ) (k)^{p_\infty} \leq C \sum_{k \in \ZZ} {g_2(\cdot)}^{\frac{p_\infty}{p(k)}} \leq C \sum_{k \in \ZZ} g_2(k) + C \sum_{k \in \ZZ} R(k)^{\frac{1}{p_\infty}} \leq C \]
Finally, note  $a(\cdot)$ is $\ell^{p_\infty} (\ZZ)$ integrable, ${g_2(\cdot)}^{\frac{1}{\pdot}}$ is also $\ell^{p_\infty} (\ZZ)$ integrable. So, we can choose $\delta_2 >0$ such that
\[ \sum_{k \in \ZZ} \delta_2 M(g_2(\cdot)^{\frac{1}{\pdot}} ) (k)^{p_\infty} + \frac{1}{6} \leq \frac{1}{3} + \frac{1}{6} = \frac{1}{2} \]
\end{proof}

\subsection{Weak $(\pdot, \pdot)$  Inequality for Fractional Hardy-Littlewood Maximal Operator}
\begin{theorem}[Weak $(\pdot, \pdot)$  Inequality for Fractional Hardy-Littlewood Maximal Operator]
Given a non-negative sequence $ \left \{ a(i) \right \} \in \ell^{\pdot}(\ZZ)$, let $ \pdot \in \mathcal{S},\pplus< \infty,  \pminus= 1 , 1 \leq \pminus \leq \pplus < \frac{1}{\alpha} \quad \text{and} \quad \pdot \in LH_\infty(\ZZ)$. Then
\[ \sup_{t > 0}  t \normb{ \chi_{ \left \{ M_{\alpha} a(k	)  >9t \right \} }}_{\lqdot (\ZZ)} \leq C \normb{a}_{\lpdot (\ZZ)} \]
\end{theorem}
\begin{proof}
Denote by $ \Omega = \left \{ k \in \ZZ: M_\alpha a(k) > 9t \right \}$.
 Using Calderon-Zygmund decomposition for sequence $\left \{ a(k) \right \}$ by Lemma[$\ref{CD_Weak}$], we get  for fixed $ t > 0$, 
\[ \left \{ k \in \ZZ: M_\alpha a(k) > 9t \right \} =   \cup_j 2I_j\]
Now, consider disjoint sets $E_j$ such that $ E_j \subset 2I_j$ and $\Omega = \cup_j E_j$.
To prove the weak inequality, it will suffice to show that for each $k \in \Omega$, $ t \normb{\chi_{\Omega}(k)}_{\qdot} \leq C$ and in turn it will suffice to show that for some $\alpha_2 >0$,
\[ \rho_{\qdot}(\alpha_2 t \chi_{\Omega} ) = \sum_{k \in \Omega}  [\alpha_2 t]^{q(m)} \leq 1 \]
We will show that each term on the right is bounded by $\frac{1}{2}$ for suitable choice of $\alpha_2$.
To estimate
$ \sum_{k \in \Omega} [\alpha_2 t]^{q(m)}$, we note from results in previous section, 
\begin{align*}
\sum_{k \in \Omega} [ \alpha_2 t]^{q(m)} & \leq \sum_{j} \sum_{E_j} \delta_2 \biggl( \frac{1}{\abs{2I_j}}  \sum_{k \in 2I_j}  {g_2(k)}^{\frac{1}{p(k)}} \biggr)^{p_\infty} + \frac{1}{6}  \\
& \leq \sum_{j} \sum_{E_j} \delta_2 \biggl( \frac{1}{\abs{2I_j}} \sum_{k \in 2I_j} {g_2(k)}^{\frac{p_\infty}{p(k)}} \biggr) + \frac{1}{6} \\
&\leq \sum_{k \in \Omega} \delta_2 {g_2(k)}^{\frac{p_\infty}{p(k)}} + \frac{1}{6}  
\end{align*}
Now , choose $\delta_2 >0$ such that right hand side is bounded by $\frac{1}{2}$.
\end{proof}

\section{ Hardy-Littlewood Maximal operator}
In this section, we prove boundedness of Hardy-Littlewood maximal operator for $\lpdot(\ZZ)$ spaces where $\pminus>1$. The proof is based on boundedness of Hardy-Littlewood maximal operator on $\ell^p(\ZZ)$, where $p$ is a fixed number, $1< p< \infty$. \\
\begin{remark}
Note that when $\alpha =0$ the fractional Hardy-Littlewood maximal operator is nothing but Hardy-Littlewood maximal operator.
However we can prove strong type, weak type inequalities for Hardy-Littlewood maximal operator on $\lpdot(\ZZ)$ directly from the corresponding results for fixed $\ell^p(\ZZ)$ spaces, $ 1 < p < \infty$. 
The key point of the proofs is Lemma[$\ref{sak_lem1}$]. The proof of continuous version of Lemma[$\ref{sak_lem1}$] can be found in $\text{\cite{fior_var_book}}$. Same line of proof works here.
\end{remark}

\begin{lemma}\label{sak_lem1}
Let $p(\cdot) : \ZZ \to [0,\infty) $ be such that $ p(\cdot) \in LH_\infty(\ZZ)$ and $ 0< p_\infty < \infty$. Let $R(k) = (e + \abs{k})^{-N}, N > \frac{1}{p_\infty} $. Then there exists a real constant C depending on N and $LH_\infty(\ZZ)$ constant of $p(\cdot)$ such that given any set E and any function F with $0\leq F(y) \leq 1$  for $ y \in E$
\begin{align*}
& \sum_E F(m)^{p(m)}  \leq \sum_E F(m)^{p_\infty} + C \sum_E R(m)^{p_\infty}  \label{eq:eq5}\tag{CD5}  \\ 
& \sum_E F(m)^{p_\infty}  \leq C \sum_E F(m)^{p(m)} + C \sum_E  R(m)^{p_\infty} \label{eq:eq6}\tag{CD6} 
\end{align*}
\end{lemma}

\subsection{Strong $(\pdot, \pdot)$ Inequality}
\begin{theorem}[Strong $(\pdot, \pdot)$ Inequality]
 Given a non-negative sequence $ \left \{ a(i) \right \} \in \ell^{\pdot}(\ZZ), \pdot \in \mathcal{S},\pplus< \infty,  \pminus> 1$ , then
\[ \normb{Ma}_{\ell^{p(.)}(\ZZ)} \leq C \normb{a}_{\ell^{p(.)}(\ZZ)} \]
\end{theorem}

\begin{proof}
By homogenity, it is enough to prove the above result with the assumption $\normb{a}_{{\ell^\pdot}(\ZZ)} =1 $.  
By Lemma$[\ref{sak_corollary_2}]$, $\sum_{i \in \ZZ} \abs{a(i)}^{p(i)} \leq 1 $. So, it is enough to prove that
\begin{align*}
\sum_{i \in \ZZ} \abs{Ma(i)}^{p(i)}  \leq C 
\end{align*}

Given that $ 0 \leq a(k) \leq 1 $, it follows that  $ 0 \leq Ma(k) \leq 1$. To prove boundedness of $ \left \{Ma \right \}$, 
we start with  Lemma[\ref{sak_lem1}] as follows:
\[ \sum_{k \in \ZZ} Ma(k)^{p(k)} \leq C \sum_{k \in \ZZ} Ma(k)^{p_\infty} + C \sum_{k \in \ZZ} R(k)^{p_\infty} \]
Since $N > \frac{1}{p_\infty}$, $ \sum_{k \in \ZZ} R(k)^{Np_\infty} = \sum_{k \in \ZZ} (\frac{1}{e+\abs{k}})^{Np_\infty}$ converges and can be bounded as $\leq 1$. So, the second integral is a constant depending only on $p_\infty$ by taking sufficiently large $N > \frac{1}{p_\infty}$.

To bound the first integral, note that  $ 1 < \pminus \leq p_\infty$. Since $p_\infty >  1$, M is bounded on $\ell^{p_\infty }(\ZZ)$ and by using strong $(p,p)$ inequality valid for classical Lebesgue spaces with index $p_\infty$, we get using Lemma[$\ref{sak_lem1}$] and [$\eqref{eq:eq5}$],
\begin{align*}
\sum_{k \in \ZZ}  Ma(k)^{p_\infty} \leq C \sum_{k \in \ZZ} a(k)^{p_\infty} \leq C \sum_{k \in\ZZ} a(k)^{p(k)} + C \sum_{k \in \ZZ} R(k)^{p_\infty} \leq C
\end{align*}
Like previous case, the term involving summation of $R(k)$ is bounded by a constant depending  only on $p_\infty$ by taking sufficiently large $N > \frac{1}{p_\infty} $.\\
Therefore, using above results,
\[ \rho_{\pdot}(Ma) = \sum_{k \in \ZZ} Ma(k)^{p(k)} \leq C. \]
\end{proof}

\subsection{Weak $(\pdot, \pdot)$  Inequality for  Hardy-Littlewood Maximal Operator}
\begin{theorem}[Weak $(\pdot, \pdot)$  Inequality for  Hardy-Littlewood Maximal Operator]
Given $\pdot \in \mathcal{S}, \pdot \geq 1$, if $\pdot \in LH_\infty(\ZZ)$, then
\[ \sup_{t > 0}  \normb{ t  \chi_{ \left \{ Ma(n) >t \right \} }}_{\lpdot (\ZZ)} \leq c \normb{a}_{\lpdot (\ZZ)} \]
where constant depends on the Log-Holder constants of $\pdot, \pminus$ and $p_\infty$(if this value is finite) .
\end{theorem}
\begin{proof}

\begin{enumerate}
\item Case: $\pminus > 1$ and $\normb{a}_{\lpdot(\ZZ)} = 1$. \\
Let $ A = \left \{ n \in \ZZ : Ma(n) > t  \right \} $. Then, by the use of strong $ (\pdot, \pdot)$ inequality for Hardy-Littlewood maximal operator from previous section, it follows that
\begin{align*}
 & \normb{ t  \chi_{ \left \{ n: Ma(n) >t \right \} }(k) }_{\lpdot (\ZZ)}  \leq \sum_{k \in \ZZ} \abs{ t \chi_{ \left  \{ n: Ma(n) > t \right \} } (k) }^{p(k)}\leq  \sum_{ k \in A} \abs{t}^{p(k)} \\
& \leq \sum_{k \in A} {Ma}^{p(k)} \leq \sum_{k \in \ZZ} {Ma(k)}^{p(k)}  = \rho_{\pdot}( Ma)  \leq C  = C \normb{a}_{\lpdot(\ZZ)} \\
\end{align*}
\item Case: $\pminus >1 $ and $\normb{a}_{\lpdot(\ZZ)} \neq 1 $ \\
Denote, $b(n)  = \frac{a(n)}{\normb{a}_{\lpdot (\ZZ)}}$ , so that $\normb{b}_{\lpdot(\ZZ)}=1$. Let $ A = \left \{ n \in \ZZ : Ma(n) > t  \right \} $
By homogenity of the norm, it follows that,
\begin{align*}
& \normb{t \chi_{ \left \{ n: Ma(n) > t \right \}}(k) }_{\lpdot(\ZZ)} \\
& = \normb{t \chi_{ \left \{ n: Mb(n) > \frac{t}{\normb{a}_{\lpdot (\ZZ)}}  \right \}}(k) }_{\lpdot(\ZZ)}  \\
& = \normb{ \normb{a}_{\lpdot (\ZZ)} \frac{t}{\normb{a}_{\lpdot (\ZZ)}} \chi_{\left \{ n: Mb(n) > \frac{t}{\normb{a}_{\lpdot (\ZZ)}} \right \} }(k) }_{\lpdot(\ZZ)} \\
& = \normb{a}_{\lpdot (\ZZ)} \normb{  \frac{t}{\normb{a}_{\lpdot (\ZZ)}} \chi_{\left \{ n: Mb(n) > \frac{t}{\normb{a}_{\lpdot (\ZZ)}} \right \} }(k) }_{\lpdot(\ZZ)} \\
& \leq C \normb{a}_{\lpdot (\ZZ)}
\end{align*}
\item Case: $\pminus=1$  and $\normb{a}_{\lpdot(\ZZ)} =1$ \\
Let $ A = \left \{ n \in \ZZ : Ma(n) > t  \right \} $. Then, by the use of strong $ (\pdot, \pdot)$ inequality for Hardy-Littlewood maximal operator from previous section, it follows that
\begin{align*}
 \normb{ t  \chi_{ \left \{ n: Ma(n) >t \right  \} }(k) }_{\lpdot(\ZZ)} & \leq \sum_{k \in \ZZ} \abs{ t \chi_{ \left  \{ n: Ma(n) < t \right \} } (k) }^{p(k)}=  \sum_{ k \in A} \abs{t}^{p(k)} 
\\
=\begin{cases}
t^{\pplus} \abs{A} & t \geq 1 \\
t^{\pminus} \abs{A} & t \leq 1
\end{cases}
\end{align*}
Now, by use of weak$(\pplus, \pplus)$ inequality applicable to $\ell^{\pplus}(\ZZ)$  spaces when $t \geq 1$ and weak$(\pplus, \pplus)$ inequality applicable to $\ell^{\pminus}(\ZZ)$ spaces when $t \leq 1$ respectively, it follows that
\[\abs{A} \leq 
\begin{cases}
\frac{C}{t^{\pplus}} \normb{a}_{\ell^\pplus (\ZZ)} &  t \geq 1 \\ \\
\frac{C}{t^{\pminus}} \normb{a}_{\ell^\pminus (\ZZ)} & t \leq 1
\end{cases}
\]
Note, by Lemma$[\ref{sak_corollary_2}]$,
\[ (\normb{a}_{\pplus})^{\pplus} \leq \sum_{k \in \ZZ} {\abs{a(k)}}^{\pplus} \leq \sum_{k \in \ZZ} {\abs{a(k)}}^{p(k)} \leq \normb{a}_{\lpdot(\ZZ)} \]
and $\normb{a}_{\pminus} = \normb{a}_1 \leq \normb{a}_{\lpdot(\ZZ)}$ as $\ell_p(\ZZ) \subset \ell_1(\ZZ)$. Hence
\begin{align*}
\abs{A} \leq 
\begin{cases}
\frac{C}{t^{\pplus}} \normb{a}_{\lpdot (\ZZ)} &  t \geq 1 \\
\frac{C}{t} \normb{a}_{\lpdot(\ZZ)} & t \leq 1
\end{cases}
\end{align*}
and therefore,
\begin{align*}
 \normb{ t  \chi_{ \left \{ n: Ma(n) >t \right  \} }(k) }_{\lpdot (\ZZ)} & = 
\begin{cases}
\leq C \normb{a}_{\lpdot (\ZZ)} & t \geq 1 \\
\leq C \normb{a}_{\lpdot (\ZZ)} & t \leq 1
\end{cases}
\end{align*}
\item Case: $\pminus=1$  and $\normb{a}_{\lpdot (\ZZ)} \neq 1$. The conclusion follows similar to case(2) and case(3).
\end{enumerate}
\end{proof}



%

\end{document}